\documentclass[11pt]{amsart}
\usepackage{amsthm}
\usepackage{mathrsfs}
\usepackage{graphicx}
\usepackage{latexsym,bm,amsmath,amssymb}
\usepackage{epsfig}
\usepackage{calligra}
\usepackage[T1]{fontenc}
\allowdisplaybreaks[4]
\title[Geometric criterion for generalized Abel equation]{A geometric criterion for equation $\dot{x}=\sum^m_{i=0}a_i(t)x^i$
having at most $m$ isolated periodic solutions}
\author[J.Huang and H.Liang]{Jianfeng Huang$^1$ and Haihua Liang$^2$}
\address{$^1$  Department of Mathematics,\ Jinan University,\ Guangzhou\ 510632,\ P.R.\ China}
\email{thuangjf@jnu.edu.cn}
\address{$^2$
Department of Computer Science,\ Guangdong Polytechnic
Normal University,\ Guangzhou\ 510665,\ P.R.\ China}
\email{haiihuaa@tom.com}
\thanks{The first author is supported by the NSF of China (No.11401255)
and the Fundamental Research Funds for the Central Universities (No.21614325).
The second author is supported by the NSF of China (No.11201086, No.11401255) and the Excellent Young Teachers Training Program for colleges and universities of Guangdong Province, China (No.Yq2013107).
}
\newtheorem{thm}{Theorem}[section]
\newtheorem{prop}[thm]{Proposition}
\newtheorem{lem}[thm]{Lemma}
\newtheorem{cor}[thm]{Corollary}

\theoremstyle{remark}

\theoremstyle{definition}

\newcommand{\wt}{\widetilde}

\newtheorem{rem}[thm]{Remark}

\begin{document}

\maketitle

\begin{abstract}
This paper is devoted to the investigation of generalized Abel equation
$\dot{x}=S(x,t)=\sum^m_{i=0}a_i(t)x^i$, where $a_i\in\mathrm C^{\infty}([0,1])$. A solution $x(t)$ is called a {\em periodic solution} if $x(0)=x(1)$.
In order to estimate the number of isolated periodic solutions of the equation, we propose a hypothesis (H) which is only concerned with $S(x,t)$ on $m$ straight lines:
There exist $m$ real numbers $\lambda_1<\cdots<\lambda_m$ such that either
$(-1)^i\cdot S(\lambda_i,t)\geq0$ for $i=1,\cdots,m$, or $(-1)^i\cdot S(\lambda_i,t)\leq0$ for $i=1,\cdots,m$. By means of Lagrange interpolation formula, we proves that the equation has at most $m$ isolated periodic solutions (counted with multiplicities) if hypothesis (H) holds, and the upper bound is sharp.
Furthermore, this conclusion is also obtained under some weaker geometric hypotheses.
 Applying our main result for the trigonometrical generalized Abel equation with coefficients of degree one,
  we give a criterion to obtain the upper bound for the number of isolated periodic solutions. This criterion is ``almost equivalent" to hypothesis (H) and can be much more effectively checked.

\end{abstract}

\tableofcontents \setcounter{tocdepth}{1}

\section{Introduction and statements of main results}\label{intro}
In this work we consider the generalized Abel equation
\begin{align}\label{eq1}
&\dot{x}=\frac{dx}{dt}=S(x,t)=\sum^m_{i=0}a_i(t)x^i,
\end{align}
where $x\in\mathbb R$ and $a_i\in\mathrm C^{\infty}([0,1])$, $i=0,1,\cdots,m$.
A solution $x(t)$ of \eqref{eq1} is called a {\em periodic solution}, if it is defined in $[0,1]$ with $x(0)=x(1)$.

Equation \eqref{eq1} is not only a powerful tool
in dealing with the problems for limit cycles of planar
differential systems (see Cherkas \cite{taubes28}, Devlin {\em et al.} \cite{taubes5}, Lins-Neto \cite{taubes3} and Lloyd \cite{taubes4}),
but also extensively applied in the biological studies
(for instance, harvesting model, see \cite{taubes21}, \cite{taubes22} and \cite{taubes23}).
Under these backgrounds,
one of the main problems for \eqref{eq1} is to estimate its number of isolated periodic solutions.

Unfortunately, even this simplest type of differential equations
is not completely understood yet.
The first motivated works on equation \eqref{eq1} are given by Lins-Neto \cite{taubes3} and Lloyd \cite{taubes15, taubes17, taubes4}.
Both authors prove that \eqref{eq1} has at most one (resp. two) isolated periodic solutions if $m=1$ (resp. $m=2$). However,
an unexpected result shown in \cite{taubes3} is that
the number of isolated periodic solutions is not bounded for \eqref{eq1} with $m=3$ (see also Panov \cite{taubes11}).
Such result is easily extended for the equation with $m>3$ (see Gasull and Guillamon \cite{taubes8}).
Therefore, in order to bound the number of isolated periodic solutions of \eqref{eq1},
some essential conditions must be imposed.

The most classical hypothesis on which many results depend is that one of the first two non-zero coefficients of \eqref{eq1} has definite sign.
For instance,
it is well-known that \eqref{eq1} has at most three isolated periodic solutions if it is Abel equation (i.e.  $m=3$) with $a_3(t)$ keeping sign
(see Gasull and Llibre \cite{taubes2}, Lins-Neto \cite{taubes3}, Lloyd \cite{taubes4} and Pliss \cite{taubes16}).
For the case that $m=3$ and $a_2(t)$ does not change sign,
the authors in paper \cite{taubes2} prove that the number of isolated periodic solutions of \eqref{eq1} is still no more than three.
Another notable result is due to Ilyashenko \cite{taubes6}.
When $m>3$ and $a_m(t)\equiv1$, he give an upper bound for the number of isolated periodic solutions
associated with the bounds of $|a_i(t)|$, $i=0,\cdots, m-1$.
Generalized Abel equations with some coefficients $a_i(t)$ of definite signs
are also investigated by several authors (see Alkoumi and Torres \cite{taubes26}, Gasull and Guillamon \cite{taubes8} and Panov \cite{taubes10}).

The hypotheses presented above are easily invalid in some cases (especially when \eqref{eq1} is of the trigonometrical type).
In recent years, two families of investigations which admit the coefficients without fixed signs are proposed. The first one requires the symmetric conditions for some coefficients of the equation (see \cite{taubes27}, \cite{taubes7} and \cite{taubes25}).
The second one depends on the fixed sign hypotheses for
some linear combinations of the coefficients. It is started and
 motivated by \'{A}lvarez, Gasull and Giacomini \cite{taubes1}. They prove that if $S(x,t)=a_3(t)x^3+a_2(t)x^2$ with
 \begin{align*}
 a\cdot a_3(t)+b\cdot a_2(t)\neq0,\indent a,b\in\mathbb R,
 \end{align*}
 then \eqref{eq1} has at most one non-zero isolated periodic solution.
Later, Huang and Zhao in \cite{taubes18} consider a generalized case $S(x,t)=a_m(t)x^m+a_n(t)x^n+a_l(t)x^l$.
Under the hypotheses of the parity of $m,n,l$ and the inequality
 \begin{align*}
 a_l(t)\cdot\left(a_m(t)\lambda^{m-n}+a_n(t)\right)\cdot S\left((\pm1)^{m-n+1}\lambda,t\right)\neq0,\indent\lambda\neq0,
 \end{align*}
 they also get the upper bound for the number of isolated periodic solutions.
 In paper \cite{taubes19}, \'{A}lvarez,  Bravo and Fern\'{a}ndez study \eqref{eq1} and give a criterion to estimate the upper bound,
 which can be computationally checked by the algebraic methods.

For more works on the periodic solutions of \eqref{eq1}, see \cite{taubes9}, \cite{taubes20}, \cite{taubes13}, \cite{taubes12}, \cite{taubes24} and \cite{taubes14}, etc.
\vskip 0.3cm

Our theorem mainly extend the results of \cite{taubes1} and \cite{taubes18}. 
More precisely,
in order to control the upper bound for the number of periodic solutions, we propose the following hypothesis for \eqref{eq1}:\\
\noindent{\bf (H)} {\em There exist $m$ real numbers $\lambda_1<\cdots<\lambda_m$ such that either
$(-1)^i\cdot S(\lambda_i,t)\geq0$ for $i=1,\cdots,m$, or $(-1)^i\cdot S(\lambda_i,t)\leq0$ for $i=1,\cdots,m$.}

\begin{thm}\label{thm1}
Assume that hypothesis {\bf (H)} holds. Then \eqref{eq1} has at most $m$ isolated periodic solutions (counted with multiplicities). This upper bound is sharp.
\end{thm}
Two examples will be provided to illustrate the direct application of Theorem \ref{thm1}. Moreover, the existence and the location of the isolated periodic solutions in these examples will also be shown. See section 3 for details.

\begin{rem}\label{rem1}
Clearly, Hypothesis {\bf (H)} allows all the coefficients $a_0(t),\cdots, a_m(t)$ of \eqref{eq1} to be non-zero (the results in \cite{taubes1} and \cite{taubes18} admit at most three non-zero coefficients of \eqref{eq1}). And for arbitrary $m$, it always provides linear conditions with respect to these coefficients. In addition, it does not refer to the signs of $a_0(t),\cdots, a_m(t)$, but is only concerned with $S(x,t)$ on $m$ straight lines $x=\lambda_1,\cdots,\lambda_m$.
\end{rem}

It is notable that
 Hilbert's 16th problem for many planar polynomial differential systems can be reduced to the problem of determining
 the maximal number of isolated periodic solutions of \eqref{eq1},
 where the coefficients $a_0(t),\cdots, a_m(t)$ are trigonometrical polynomials (see Lins-Neto \cite{taubes3}).
  On the other hand, the cases with unbounded numbers of periodic solutions given in \cite{taubes3} and \cite{taubes11} are also
  trigonometrical generalized Abel equations.
Those numbers of periodic solutions increase with respect to the degrees of the trigonometrical
coefficients.
For these reasons, Lins \cite{taubes3} and Ilyashenko \cite{taubes14} propose to find the bound for the number of periodic solutions of trigonometrical generalized Abel equation
in terms of the degrees of the coefficients.
As an application of Theorem \ref{thm1}, we
study a simple case of \eqref{eq1} with trigonometrical coefficients of degree one.

\begin{cor}\label{example1}
Consider differential equation
\begin{align}\label{eq28}
\begin{split}
\frac{dx}{dt}=S(x,t)=&\sum^m_{i=0}\bigg(a_i+b_i\cos(2\pi t)+c_i\sin(2\pi t)\bigg)x^i,
\end{split}
\end{align}
where $t\in[0,1]$ and $a_i, b_i, c_i\in\mathbb R$, $i=0,\cdots,m$. Let
\begin{align}\label{eq30}
 f_{\text{\bf{\em a}}}(x)\triangleq \sum^m_{i=0}a_ix^i,\ \ f_{\text{\bf{\em b}}}(x)\triangleq \sum^m_{i=0}b_ix^i,\ \ f_{\text{\bf{\em c}}}(x)\triangleq \sum^m_{i=0}c_ix^i.
\end{align}
 Suppose that the following condition holds:
\begin{itemize}
\item[{\bf (C)}] There exist $m$ zeros $\kappa_1<\cdots<\kappa_m$ of $f_{\text{\bf{\em a}}}^2(x)-f_{\text{\bf{\em b}}}^2(x)-f_{\text{\bf{\em c}}}^2(x)$ such that either
$(-1)^i\cdot f_{\text{\bf{\em a}}}(\kappa_i)\geq0$ for $i=1,\cdots,m$, or $(-1)^i\cdot f_{\text{\bf{\em a}}}(\kappa_i)\leq0$ for $i=1,\cdots,m$.
\end{itemize}
Then \eqref{eq28} has
at most $m$ isolated periodic solutions (counted with multiplicities).
\end{cor}

An example with all the trigonometrical coefficients changing signs will be given after the proof of Corollary \ref{example1}, see section 4 for details.

As a sufficient condition which implies Hypothesis {\bf (H)}, condition {\bf (C)} in Corollary \ref{example1} is algebraic and can be much more effectively checked (at most $2m$ real zeros of $f_{\text{\bf{\em a}}}^2(x)-f_{\text{\bf{\em b}}}^2(x)-f_{\text{\bf{\em c}}}^2(x)$ need to be checked). Furthermore, in some sense, it is also ``almost necessary'' for Hypothesis {\bf (H)} when studying \eqref{eq28}. In fact, we have

\begin{prop}\label{prop1}
Let $S(x,t)$ be defined as in \eqref{eq28}. Let $f_{\text{\bf{\em a}}}, f_{\text{\bf{\em b}}}, f_{\text{\bf{\em c}}}$ be defined as in \eqref{eq30}. Assume that equation \eqref{eq28} satisfies Hypothesis {\bf (H)}. Then it either has constant periodic solutions, or satisfies condition {\bf (C)}.
\end{prop}

In section 5 we show that the conclusion in Theorem \ref{thm1} is still valid under some weaker geometric conditions. More precisely, let $a(t), b(t)$ be two smooth functions on interval $[0,1]$ with $a(0)=a(1), b(0)=b(1)$ and $a(t)>0$. Let
\begin{align*}
\gamma_{\lambda}:\big(t,\lambda a(t)+b(t)\big)
\end{align*}
be a family of curves lying in $[0,1]\times\mathbb R$ with parameter $\lambda\in\mathbb R$,
and $$v_S=\big(1,S(x,t)\big)$$ be the vector field induced by \eqref{eq1}.
We give the next hypothesis.

\noindent{\bf (H')} {\em There exist $m$ real numbers $\lambda_1<\cdots<\lambda_m$ such that either
$(-1)^i\cdot\det\left(\dot{\gamma_{\lambda_i}}, v_S\right)\left|_{\gamma_{\lambda_i}}\right.\geq0$ for $i=1,\cdots,m$,
or $(-1)^i\cdot\det\left(\dot{\gamma_{\lambda_i}}, v_S\right)\left|_{\gamma_{\lambda_i}}\right.\leq0$ for $i=1,\cdots,m$, where $\dot{}$ represents the first-order derivative and $\det(\cdot,\cdot)$ is the
determinant of two $2$-dimensional vector fields.}

\begin{thm}\label{thm2}
  Assume that Hypothesis {\bf (H')} holds.
  Then \eqref{eq1} has at most $m$ non-zero isolated periodic solutions (counted with multiplicities). This upper bound is sharp.
\end{thm}

Clearly, two planar vectors are positively oriented (resp. negatively oriented) if and only if the determinant is positive (resp. negative). This leads to the following conclusion.
\begin{cor}\label{cor1}
  Assume that $v_S$ is transverse to $m$ curves $\gamma_{\lambda_1},\cdots,\gamma_{\lambda_m}$, where $\lambda_1<\cdots<\lambda_m$. If $\{v_S,\dot{\gamma_{\lambda_i}}\}$ and $\{v_S,\dot{\gamma_{\lambda_{i+1}}}\}$ have opposite orientations for each $i=1,\cdots,m-1$, then \eqref{eq1} has at most $m$ non-zero isolated periodic solutions (counted with multiplicities). This upper bound is sharp.
\end{cor}

The rest of this paper is organized as follows: In section 2 we
give several preliminary results. In section 3 we
prove Theorem \ref{thm1}. Corollary \ref{example1} and Proposition \ref{prop1} are obtained in section 4. Theorem \ref{thm2} and Corollary \ref{cor1} are proved in section 5.

\section{Preliminaries}\label{an
apprpriate label}
In this section we mainly give two lemmas and two properties that are useful for the proofs of the theorems.
\vskip0.3cm


Let $x(t,x_0)$ be the solution of \eqref{eq1} with $x(0,x_0)=x_0$. It is well-known that (see Lloyd \cite{taubes4} for instance)
\begin{align}\label{eq7}
\frac{\partial x}{\partial x_0}(t,x_0)=\exp\int^t_0\frac{\partial S}{\partial x}\big(x(s,x_0),s\big)ds.
\end{align}
Moreover,
for the return map
\begin{align*}
H(x_0)=x(1,x_0),
\end{align*}
we have
\begin{align}\label{eq40}
\begin{split}
&\dot{H}(x_0)=\exp\int^1_0\frac{\partial S}{\partial x}\big(x(t,x_0),t\big)dt,\\
&\ddot{H}(x_0)=\dot{H}(x_0)\cdot\int^1_0\frac{\partial^2 S}{\partial x^2}\big(x(t,x_0),t\big)
\cdot\frac{\partial x}{\partial x_0}(t,x_0)dt,
\end{split}
\end{align}
where $\dot{}$ and $\ddot{}$ represent the first-order and second-order derivatives, respectively.
Thus the lemma below is obtained.
\begin{lem}\label{lem4}
Let $U$ be an open interval in $\mathbb R$ and $S(x,t)$ be defined as in \eqref{eq1}. Suppose that $F\in\mathrm C^0(U)$ and
\begin{align*}
G(x,t)\triangleq\frac{\partial S}{\partial x}(x,t)+F(x)\cdot S(x,t),\indent (x,t)\in U\times([0,1]).
\end{align*}
The following statements hold.
\begin{itemize}
\item[(i)] Assume that $x(t,x_0)$ is a periodic solution of \eqref{eq1} located in $U$.
Then
\begin{align}\label{eq39}
\begin{split}
\int^{1}_{0}\frac{\partial S}{\partial x}\big(x(t,x_0),t\big)dt=\int^{1}_{0}G\big(x(t,x_0),t\big)dt.
\end{split}
\end{align}
\item[(ii)] If $G\big|_{U\times[0,1]}\geq0$ (resp. $\leq0$) and there exists a non-empty open set $E\subseteq[0,1]$ such that
$G\big|_{U\times E}\neq0$,
then \eqref{eq1} has at most $1$ periodic solution in $U$, which is hyperbolic unstable (resp. stable).
\end{itemize}
\end{lem}
\begin{proof}
(i) By assumption, we know $x(t,x_0)\in U$ for $t\in[0,1]$, and $x(0,x_0)=x(1,x_0)$. Therefore,
\begin{align*}
\int^{1}_{0}F\big(x(t,x_0)\big)\cdot S\big(x(t,x_0),t\big)dt
&=\int^{x(1,x_0)}_{x(0,x_0)}F(x)dx
=0,
\end{align*}
which implies that \eqref{eq39} holds.

(ii) 
For an arbitrary periodic solution $x(t,x_0)$ in $U$, it follows from \eqref{eq40} and statement (i) that
\begin{align*}
\dot{H}(x_0)
&=\exp\left(\int_{[0,1]\backslash E}G\big(x(t,x_0),t\big)dt\right)
\cdot\exp\left(\int_{E}G\big(x(t,x_0),t\big)dt\right).
\end{align*}
Hence, $x(t,x_0)$ is hyperbolic unstable (resp. stable) if $G|_{U\times[0,1]}\geq0$ (resp. $\leq0$) and $G|_{U\times E}\neq0$.
Together with the fact that two consecutive hyperbolic
periodic solutions have different stability (see also in \cite{taubes8}), \eqref{eq1} has at most 1 periodic solution in $U$.
\end{proof}

Suppose that \eqref{eq1} satisfies Hypothesis {\bf (H)}. Then a periodic solution of \eqref{eq1} is either located in $\mathbb R\backslash\{\lambda_1,\cdots,\lambda_m\}$, or equal identically to one of $\lambda_1,\cdots,\lambda_m$. Now define
\begin{align}\label{eq32}
f(x)=\prod_{i=1}^{m}(x-\lambda_i).
\end{align}
We know by \eqref{eq1} that $S(x,t)-a_m(t)f(x)$ is a polynomial in $x$ of degree no more than $m-1$.
According to Lagrange interpolation formula, for $x\in \mathbb R\backslash\{\lambda_1,\cdots,\lambda_m\}$,
\begin{align*}
S(x,t)-a_m(t)f(x)
&=\sum^{m}_{i=1}\frac{\prod_{j\neq i}(x-\lambda_j)}{\prod_{j\neq i}(\lambda_i-\lambda_j)}\cdot S(\lambda_i,t)\\
&=\left(\sum^{m}_{i=1}\frac{S(\lambda_i,t)}{\prod_{j\neq i}(\lambda_i-\lambda_j)}\cdot\frac{1}{x-\lambda_i}\right)\cdot f(x),
\end{align*}
i.e.
\begin{align}\label{eq45}
S(x,t)=\left(a_m(t)+\sum^{m}_{i=1}\frac{S(\lambda_i,t)}{\prod_{j\neq i}(\lambda_i-\lambda_j)}\cdot
\frac{1}{x-\lambda_i}\right) f(x).
\end{align}
Hence, together with Lemma \ref{lem4}, we get the following proposition.


\begin{prop}\label{lem5}
Assume that \eqref{eq1} satisfies Hypothesis {\bf (H)}. If $S(\lambda_1,t),\cdots,S(\lambda_m,t)$ are not
all identically zero, then the following statements hold.
\begin{itemize}
\item[(i)] Equation \eqref{eq1} has at most $1$ isolated periodic solution in each connected component of $\mathbb R\backslash\{\lambda_1,\cdots,\lambda_m\}$, counted with multiplicity.
\item[(ii)] Assume that $\lambda\in\{\lambda_1,\cdots,\lambda_m\}$. Let $U_1$ and $U_2$ be two different connected components of $\mathbb R\backslash\{\lambda_1,\cdots,\lambda_m\}$, with point $\lambda$ adjoined. If $S(\lambda,t)\equiv0$, then either $U_1$ or $U_2$ contains no periodic solutions of \eqref{eq1}. Furthermore, the multiplicity of $x(t,\lambda)\equiv\lambda$ is no more than $2$.
\end{itemize}
\end{prop}
\begin{proof}
Firstly, denote by
\begin{align*}
I_S(x,t)=\sum^{m}_{i=1}\frac{S(\lambda_i,t)}{\prod_{j\neq i}(\lambda_i-\lambda_j)}\cdot\frac{1}{(x-\lambda_i)^2},
\indent x\in\mathbb R\backslash\{\lambda_1,\cdots,\lambda_m\},\ t\in[0,1].
\end{align*}
Let $U$ be an arbitrary connected component of $\mathbb R\backslash\{\lambda_1,\cdots,\lambda_m\}$, and let $f(x)$ be defined as in \eqref{eq32}.
Taking
\begin{align*}
F(x)=-\frac{\dot{f}(x)}{f(x)}=-\sum^{m}_{i=1}\frac{1}{x-\lambda_i},
\indent x\in U
\end{align*}
in Lemma \ref{lem4}, we have by \eqref{eq45} that
\begin{align}\label{eq44}
\begin{split}
&G(x,t)=\frac{\partial S}{\partial x}(x,t)+F(x)\cdot S(x,t)\\
&\ \ \indent\indent
=-\left(\sum^{m}_{i=1}\frac{S(\lambda_i,t)}{\prod_{j\neq i}(\lambda_i-\lambda_j)}\cdot\frac{1}{(x-\lambda_i)^2}\right)\cdot f(x)\\
&\ \ \indent\indent\indent
+\left(a_m(t)+\sum^{m}_{i=1}\frac{S(\lambda_i,t)}{\prod_{j\neq i}(\lambda_i-\lambda_j)}\cdot\frac{1}{x-\lambda_i}\right)\cdot\dot{f}(x)\\
&\ \ \indent\indent\indent
-\left(a_m(t)+\sum^{m}_{i=1}\frac{S(\lambda_i,t)}{\prod_{j\neq i}(\lambda_i-\lambda_j)}\cdot\frac{1}{x-\lambda_i}\right)\cdot\dot{f}(x)\\
&\ \ \indent\indent
=-I_S(x,t)\cdot f(x).
\end{split}
\end{align}

Observe that $\lambda_1<\cdots<\lambda_m$ from Hypothesis {\bf (H)}. For each fixed $i\in\{1,\cdots,m\}$, we obtain
\begin{align*}
\#\{\lambda_j|\lambda_j>\lambda_i, j=1,\cdots,m\}=m-i,
\end{align*}
where $\#$ represents the cardinality of a set.
This implies that
\begin{align}\label{eq50}
\begin{split}
\text{sgn}\left(\frac{S(\lambda_i,t)}{\prod_{j\neq i}(\lambda_i-\lambda_j)}\right)
=\frac{\text{sgn}\left((-1)^i S(\lambda_i,t)\right)}{(-1)^i\prod_{j\neq i}\text{sgn}(\lambda_i-\lambda_j)}
=(-1)^{m}\cdot\text{sgn}\left((-1)^i S(\lambda_i,t)\right).
\end{split}
\end{align}
Thus, Hypothesis {\bf (H)} tells us that
\begin{align}\label{eq34}
I_S\left|_{(\mathbb R\backslash\{\lambda_1,\cdots,\lambda_m\})\times[0,1]}\geq0(\leq0)\right.,\ \
I_S\left|_{(\mathbb R\backslash\{\lambda_1,\cdots,\lambda_m\})\times E}\right.\neq0,
\end{align}
where
$$E=\big\{t\big|S(\lambda_1,t),\cdots,S(\lambda_m,t)\ \text{are\ not\ all\ zero}\big\}.$$
Furthermore, since $S(\lambda_1,t),\cdots,S(\lambda_m,t)$ are not
all identically zero, $E$ is a non-empty open set.
\vskip 0.3cm

Now we prove statements (i) and (ii) one by one.

(i) Clearly, $f|_U\neq0$. It follows from \eqref{eq44} and \eqref{eq34} that $G|_{U\times[0,1]}\geq0(\leq0)$ and $G|_{U\times E}\neq0$.
According to statement (ii) of Lemma \ref{lem4},
\eqref{eq1} has at most $1$ isolated periodic solutions in $U$, counted with multiplicities. As a result, statement (i) is valid.

 (ii) From assumption, there exists $\lambda_p\in\{\lambda_1,\cdots,\lambda_m\}$ such that $S(\lambda_p,t)\not\equiv0$. In what follows we consider a perturbation of \eqref{eq1}
\begin{align}\label{eq12}
\begin{split}
\frac{dx}{dt}&=S_{\varepsilon_1,\varepsilon_2}(x,t)\\
&=S(x,t)+\varepsilon_1 f(x)+\varepsilon_2 S(\lambda_p,t)\cdot(x-\lambda)^2\prod_{\lambda_j\neq\lambda,\lambda_p}(x-\lambda_j),
\end{split}
\end{align}
where
\begin{align*}
 \varepsilon_1\in\mathbb R,\indent
 \left|\varepsilon_2\right|<\left|\frac{1}{(\lambda_p-\lambda)^2\prod_{\lambda_j\neq\lambda,\lambda_p}(\lambda_p-\lambda_j)}\right|.
 \end{align*}
Observe that
\eqref{eq12} is of the form \eqref{eq1} with
\begin{align*}
&S_{\varepsilon_1,\varepsilon_2}(\lambda_i,t)=S(\lambda_i,t),\indent \lambda_i\in\{\lambda_1,\cdots,\lambda_m\}\backslash\{\lambda_p\},\\
&S_{\varepsilon_1,\varepsilon_2}(\lambda_p,t)=\left(1
+\varepsilon_2(\lambda_p-\lambda)^2\prod_{\lambda_j\neq\lambda,\lambda_p}(\lambda_p-\lambda_j)\right)S(\lambda_p,t).
\end{align*}
It also satisfies Hypothesis {\bf (H)} with $S_{\varepsilon_1,\varepsilon_2}(\lambda,t)\equiv0$ and $S_{\varepsilon_1,\varepsilon_2}(\lambda_p,t)\not\equiv0$. Therefore,
Statement (i) is usable for \eqref{eq12}.

Now assume for a contradiction that there exist two periodic solutions $x(t,x_1)$ and $x(t,x_2)$ of \eqref{eq1} (i.e. \eqref{eq12}$|_{\varepsilon_1=0,\varepsilon_2=0}$) contained in $U_1$ and $U_2$, respectively.
Since $U_1$, $\{\lambda\}$ and $U_2$ are three consecutive sets by assumption, $f|_{U_1}$ and $f|_{U_2}$ have opposite signs. According to \eqref{eq44}, \eqref{eq34} and statement (ii) of Lemma \ref{lem4} again,
$x(t,x_1)$, $x(t,x_2)$ are hyperbolic with different stabilities.

As a result,
\eqref{eq12} has three periodic solutions in $U_1\cup\{\lambda\}\cup U_2$ as $\varepsilon_1$ and $\varepsilon_2$ are small enough, where the middle one is
$x(t)=x(t,\lambda)\equiv\lambda$ and the other two are hyperbolic with different stabilities.
In addition, statement (i) tells us that these three periodic solutions are consecutive, which implies that $x(t)\equiv\lambda$ is always semi-stable for \eqref{eq12} as $\varepsilon_1$ and $\varepsilon_2$ are small enough.

However, a direct calculation shows that
\begin{align*}
\begin{split}
\frac{\partial S_{\varepsilon_1,\varepsilon_2}}{\partial x}(\lambda,t)
&=\frac{\partial S}{\partial x}(\lambda,t)+\varepsilon_1\dot{f}(x)\\
&=\frac{\partial S}{\partial x}(\lambda,t)+\varepsilon_1\prod_{\lambda_j\neq\lambda}(\lambda-\lambda_j).
\end{split}
\end{align*}
For the return map $H_{\varepsilon_1,\varepsilon_2}$ of \eqref{eq12}, we get by \eqref{eq40} that
\begin{align}\label{eq35}
\begin{split}
\dot{H}_{\varepsilon_1,\varepsilon_2}(\lambda)
&=\exp\left(\int^1_0\frac{\partial S}{\partial x}(\lambda,t)dt+\varepsilon_1\prod_{\lambda_j\neq\lambda}(\lambda-\lambda_j)\right)\\
&=\dot{H}_{0,0}(\lambda)\cdot\exp\left(\varepsilon_1\prod_{\lambda_j\neq\lambda}(\lambda-\lambda_j)\right).
\end{split}
\end{align}
Since $\prod_{\lambda_j\neq\lambda}(\lambda-\lambda_j)\neq0$, there always exists a sufficiently small $\varepsilon_0$ such that $\dot{H}_{\varepsilon_0,\varepsilon_2}(\lambda)\not=1$, i.e. $x(t)\equiv\lambda$ is hyperbolic for \eqref{eq12}$|_{\varepsilon_1=\varepsilon_0}$. This shows a contradiction. Therefore, equation \eqref{eq1} has no periodic solution in either $U_1$ or $U_2$.

We continue to use \eqref{eq12} to prove the rest part of statement (ii). Assume for a contradiction that $x(t)=x(t,\lambda)\equiv\lambda$ is a periodic solution of \eqref{eq1} (i.e. \eqref{eq12}$|_{\varepsilon_1=0,\varepsilon_2=0}$) with
multiplicity at least $3$. Then
$\dot H_{0,0}(\lambda)=1$ and $\ddot H_{0,0}(\lambda)=0$.
Hence for arbitrary $\varepsilon_2$, it follows from \eqref{eq35} that
\begin{align}\label{eq46}
\dot H_{0,\varepsilon_2}(\lambda)=\dot H_{0,0}(\lambda)=1.
\end{align}
Furthermore, together with \eqref{eq7}, \eqref{eq40} and \eqref{eq12}, we get
\begin{align}\label{eq47}
\begin{split}
 \ddot{H}_{0,\varepsilon_2}(\lambda)
 &=\dot H_{0,\varepsilon_2}(\lambda)\cdot
 \int^1_0\frac{\partial^2 S_{0,\varepsilon_2}}{\partial x^2}(\lambda,t)
 \cdot\exp\left(\int^t_0\frac{\partial S_{0,\varepsilon_2}}{\partial x}(\lambda,s)ds\right)dt\\
 &=\int^1_0\left(\frac{\partial^2 S}{\partial x^2}(\lambda,t)+2\varepsilon_2 S(\lambda_p,t)\prod_{\lambda_j\neq\lambda,\lambda_p}(\lambda-\lambda_j)\right)
\cdot\exp\left(\int^t_0\frac{\partial S}{\partial x}(\lambda,s)ds\right)dt\\
&=\ddot{H}_{0,0}(\lambda)+\int^1_0 2\varepsilon_2 S(\lambda_p,t)\prod_{\lambda_j\neq\lambda,\lambda_p}(\lambda-\lambda_j)
\cdot\exp\left(\int^t_0\frac{\partial S}{\partial x}(\lambda,s)ds\right)dt\\
&=2\varepsilon_2\prod_{\lambda_j\neq\lambda,\lambda_p}(\lambda-\lambda_j)\int^1_0 S(\lambda_p,t)
\cdot\exp\left(\int^t_0\frac{\partial S}{\partial x}(\lambda,s)ds\right)dt.
\end{split}
 \end{align}

Without loss of generality, suppose that $x(t)\equiv\lambda$
is stable from above for \eqref{eq1} (i.e. \eqref{eq12}$|_{\varepsilon_1=0,\varepsilon_2=0}$).
Note that
$S(\lambda_p,t)\prod_{\lambda_j\neq\lambda,\lambda_p}(\lambda-\lambda_j)\not\equiv0$.
By \eqref{eq47} we can choose a small $\varepsilon'_2$ such that $\ddot{H}_{0,\varepsilon'_2}(\lambda)>0$.
 Together with \eqref{eq46}, $x(t)\equiv\lambda$ becomes unstable from above, and therefore \eqref{eq12}$|_{\varepsilon_1=0,\varepsilon_2=\varepsilon'_2}$ has a periodic solution near and above $x(t)\equiv\lambda$. Recall that statement (i) is usable for \eqref{eq12}. This periodic solution is hyperbolic.

 However, following \eqref{eq35} again, $\dot H_{\varepsilon_1,\varepsilon'_2}(\lambda)<1$ for either $\varepsilon_1>0$ or $\varepsilon_1<0$. Thus there exists a small $\varepsilon'_1$ such that \eqref{eq12}$|_{\varepsilon_1=\varepsilon'_1,\varepsilon_2=\varepsilon'_2}$ has a second periodic solution near and above $x(t)\equiv\lambda$. As a result, there exist two periodic solutions of \eqref{eq12}$|_{\varepsilon_1=\varepsilon'_1,\varepsilon_2=\varepsilon'_2}$ contained in one connected components of $\mathbb R\backslash\{\lambda_1,\cdots,\lambda_m\}$. This contradicts to statement (i) for \eqref{eq12}. Based on the above argument, the multiplicity of $x(t)=x(t,\lambda)\equiv\lambda$ for \eqref{eq1} is no more than $2$.
\end{proof}

\begin{lem}\label{lem2}
Let $L\in\mathrm C^{1}(\mathbb R)$ and $p\in\mathrm C^{1}([0,1])$. The following statements hold.
\begin{itemize}
\item[(i)] If $x(t)$ is a solution of differential equation
\begin{align}\label{eq24}
\frac{dx}{dt}=p(t)L(x),
\end{align}
then for a point $t$ in the domain,
\begin{align}\label{eq21}
\mathrm{sgn}\big(x(t)-x(0)\big)
&=\mathrm{sgn}\left(\int^{t}_{0}p(s)ds\right)\cdot\mathrm{sgn}\bigg(L\big(x(0)\big)\bigg).
\end{align}
\item[(ii)] Assume that $E$ is a non-empty set in $[0,1]$, $q\in\mathrm C^{1}\big(\mathbb R\times[0,1]\big)$ and
\begin{align*}
 q\big|_{(a,b)\times[0,1]}\geq0(\leq0),\ q\big|_{(a,b)\times E}\neq0,
\end{align*}
where either $L(a)=0$ or $L(b)=0$. If
\begin{align}\label{eq22}
&\mathrm{sgn}\left(\int^1_0p(s)ds\right)\cdot\mathrm{sgn}\big(L(x)\big)\cdot\mathrm{sgn}\big(q(x,t)\big)\geq0,
&(x,t)\in(a,b)\times E,
\end{align}
then the differential equation
\begin{align}\label{eq23}
\frac{dx}{dt}=p(t)L(x)+q(x,t)
\end{align}
has no periodic solution in $(a,b)$. Furthermore, an arbitrary solution $\wt x(t)$ of \eqref{eq23}, with $\wt x(t)\in(a,b)$ for $t\in[0,1]$, satisfies
\begin{align}\label{eq51}
&\mathrm{sgn}\big(\wt x(1)-\wt x(0)\big)
=\mathrm{sgn}\big(q(x,t)\big),\indent (x,t)\in(a,b)\times E.
\end{align}
\end{itemize}
\end{lem}

\begin{proof}
(i) Clearly if $L\big(x(0)\big)=0$, then $x(t)\equiv x(0)$. Equation \eqref{eq21} holds. If $L\big(x(0)\big)\neq0$, then $L\big(x(t)\big)\neq0$, i.e. $L\big(x(t)\big)$ does not change sign.
As a result, we have
\begin{align*}
\int^{x(t)}_{x(0)}\frac{dx}{L(x)}&=\int^{t}_{0}p(s)ds,
\end{align*}
which also implies equation \eqref{eq21}.

(ii) We suppose $L(a)=0$ and $q(x,t)\geq0$ in $(a,b)\times[0,1]$, without loss of generality (it is a similar argument for the rest cases). Let $\wt x(t)$ be defined as in assumption. Comparing \eqref{eq23} with \eqref{eq24}, we know that if $x(t)$ is a solution of \eqref{eq24} with $x(0)=\wt x(0)>a$, then $\wt x(t)\geq x(t)>a$ for $t\in [0,1]$. Observe that $q\big|_{(a,b)\times E}>0$. $\wt x(1)>x(1)$ is obtained. Moreover, it follows from \eqref{eq22} that
\begin{align*}
\mathrm{sgn}\left(\int^1_0p(s)ds\right)\cdot\mathrm{sgn}\big(L(x)\big)\geq0,\indent x\in(a,b).
\end{align*}
Combining \eqref{eq21} we get $\mathrm{sgn}\big(x(1)-x(0)\big)\geq0$. This implies that $\mathrm{sgn}\big(\wt x(1)-\wt x(0)\big)>0$ and
\eqref{eq51} hold. As a result, \eqref{eq23} has no periodic solution in $(a,b)$.
\end{proof}

\begin{prop}\label{lem3}
Assume that \eqref{eq1} satisfies Hypothesis {\bf (H)}. The following statements hold.
\begin{itemize}
\item[(i)] If $S(\lambda_1,t),\cdots,S(\lambda_m,t)$ are all identically zero, then
\eqref{eq1} either has no isolated periodic solutions, or has exactly $m$ isolated periodic solutions $x(t,\lambda_1)\equiv\lambda_1$, $\cdots$, $x(t,\lambda_m)\equiv\lambda_m$, counted with multiplicity.
\item[(ii)] If $S(\lambda_1,t),\cdots,S(\lambda_m,t)$ are not all identically zero, then \eqref{eq1} has no periodic solutions in either $(-\infty,\lambda_1)$ or $(\lambda_m\,+\infty)$.
\end{itemize}
\end{prop}
\begin{proof}
Let $a_m(t)$ and $f(x)$ be defined as in \eqref{eq1} and \eqref{eq32}, respectively.

(i) If $S(\lambda_1,t),\cdots,S(\lambda_m,t)$ are all identically zero, then \eqref{eq45} tells us that $S(x,t)=a_m(t)f(x)$. Using statement (i) of Lemma \ref{lem2}, we obtain the statements below.
\begin{itemize}
\item[(a)] When $\int^1_0a_m(t)dt=0$, a solution of \eqref{eq1} is periodic if it is well-defined in $[0,1]$.
\item[(b)] When $\int^1_0a_m(t)dt\neq0$, \eqref{eq1} has exactly $m$ periodic solutions $x(t,\lambda_1)\equiv\lambda_1$, $\cdots$, $x(t,\lambda_m)\equiv\lambda_m$.
\end{itemize}
In addition, since $\lambda_1,\cdots,\lambda_m$ are $m$ consecutive simple zeros of $f(x)$, we get $\dot f(\lambda_i)\neq0$ for $i=1,\cdots,m$. Hence for the case in statement (b) and the return map $H$ of \eqref{eq1},
it follows from \eqref{eq40} that
$$\dot H(\lambda_i)=\exp\left(\dot f(\lambda_i)\int^1_0a_m(t)dt\right)\neq1,\indent i=1,\cdots,m,$$
i.e. each $x(t,\lambda_i)\equiv\lambda_i$ is hyperbolic. As a result, statement (i) is valid.

(ii) First, if $S(\lambda_1,t),\cdots,S(\lambda_m,t)$ are not all identically zero, then the set $$E=\big\{t\big|S(\lambda_1,t),\cdots,S(\lambda_m,t)\ \text{are\ not\ all\ zero}\big\}$$ is non-empty.

Second, define
\begin{align*}
R_S(x,t)=\sum^{m}_{i=1}\frac{S(\lambda_i,t)}{\prod_{j\neq i}(\lambda_i-\lambda_j)}\cdot
\frac{1}{x-\lambda_i},\indent x\in\mathbb R\backslash\{\lambda_1,\cdots,\lambda_m\},\ t\in[0,1].
\end{align*}
According to \eqref{eq45}, we rewrite \eqref{eq1} in
\begin{align}\label{eq48}
\dot{x}=S(x,t)=a_m(t)f(x)+R_S(x,t) f(x)
\end{align}
and then compare it with $\dot x=a_m(t)f(x)$.
Similar to the proof of Proposition \ref{lem5}, we know by $\lambda_1<\cdots<\lambda_m$ that \eqref{eq50} holds, i.e.
\begin{align*}
\begin{split}
\text{sgn}\left(\frac{S(\lambda_i,t)}{\prod_{j\neq i}(\lambda_i-\lambda_j)}\right)
=(-1)^{m}\cdot\text{sgn}\left((-1)^i S(\lambda_i,t)\right),\indent i=1,\cdots,m.
\end{split}
\end{align*}
Therefore from Hypothesis {\bf (H)},
\begin{align*}
\text{sgn}\big(R_S(x,t)\big)
=\left\{
\begin{aligned}
&\text{sgn}\left((-1)^{m}\cdot\sum^{m}_{i=1}(-1)^{i}S(\lambda_i,t)\right), &x>\lambda_m,\\
&-\text{sgn}\left((-1)^{m}\cdot\sum^{m}_{i=1}(-1)^{i}S(\lambda_i,t)\right), &x<\lambda_1.
\end{aligned}
\right.
\end{align*}
Without loss of generality, suppose that $(-1)^{m+i} S(\lambda_i,t)\geq0$ for $i=1,\cdots,m$ (it is similar for the opposite case).
Then
\begin{align*}
R_S\left|_{(-\infty,\lambda_1)\times[0,1]}\leq0\right.,\ \
R_S\left|_{(\lambda_m\,+\infty)\times[0,1]}\geq0\right.,\ \
R_S\left|_{((-\infty,\lambda_1)\cup(+\infty,\lambda_m))\times E}\right.\neq0.
\end{align*}
Thus, we have
\begin{align*}
R_Sf\left|_{(-\infty,\lambda_1)\times[0,1]}\leq0(\geq0)\right.,\ \
R_Sf\left|_{(\lambda_m\,+\infty)\times[0,1]}\geq0\right.,\ \
R_Sf\left|_{((-\infty,\lambda_1)\cup(+\infty,\lambda_m))\times E}\right.\neq0,
\end{align*}
and
\begin{align*}
\begin{split}
\text{sgn}&\left(\int_0^1a_m(t)dt\right)\cdot\text{sgn}\big(f(x)\big)\cdot
\text{sgn}\big(R_S(x,t) f(x)\big)\\
&=\left\{
\begin{aligned}
&\text{sgn}\left(\int_0^1a_m(t)dt\right), &x>\lambda_m,\ t\in E,\\
&-\text{sgn}\left(\int_0^1a_m(t)dt\right), &x<\lambda_1,\ t\in E.
\end{aligned}
\right.
\end{split}
\end{align*}
According to \eqref{eq48} and statement (ii) of Lemma \ref{lem2},
\eqref{eq1} has no periodic solutions located in either $(-\infty,\lambda_1)$ or $(\lambda_m\,+\infty)$.
\end{proof}

\section{Proof of Theorem \ref{thm1}}\label{an
apprpriate label}
By virtue of the results given in the previous sections, we now begin to prove Theorem \ref{thm1}.

\begin{proof}[Proof of Theorem \ref{thm1}]
According to statement (i) of Proposition \ref{lem3}, the assertion is valid if $S(\lambda_1,t)=\cdots=S(\lambda_m,t)\equiv0$.
In what follows we prove the case that $S(\lambda_1,t),$ $\cdots,$ $S(\lambda_m,t)$ are not all identically zero.

Firstly, by assumption and Proposition \ref{lem5}, each periodic solution of \eqref{eq1} is either located in $\mathbb R\backslash\{\lambda_1,\cdots,\lambda_m\}$ with multiplicity $1$, or equal identically to one of $\lambda_1,\cdots,\lambda_m$ with multiplicity at most $2$.

Secondly, let $f(x)$ be defined as in \eqref{eq32}. Consider the following perturbation of \eqref{eq1}
\begin{align}\label{eq49}
\begin{split}
\frac{dx}{dt}&=\mathcal{S}_{\varepsilon_1,\varepsilon_2}(x,t)\\
&=S(x,t)+\varepsilon_1 f(x)+\varepsilon_2\dot f(x).
\end{split}
\end{align}
We claim that if $x(t)\equiv\lambda\in\{\lambda_1,\cdots,\lambda_m\}$ is a periodic solution of \eqref{eq1} (i.e. \eqref{eq49}$|_{\varepsilon_1=0,\varepsilon_2=0}$) with  multiplicity $2$, then the statements below hold for \eqref{eq49} as $\varepsilon_1\neq0$ is sufficiently small and $\varepsilon_2=0$.
\begin{itemize}
\item[(a)] $x(t)\equiv\lambda$ is a hyperbolic periodic solution of \eqref{eq49}.
\item[(b)] There exists a hyperbolic periodic solution of \eqref{eq49} near $x(t)\equiv\lambda$.
\end{itemize}
In fact, a direct calculation shows that
\begin{align}\label{eq13}
\begin{split}
&\mathcal{S}_{\varepsilon_1,0}(\lambda_i,t)=S(\lambda_i,t),\\
&\frac{\partial \mathcal{S}_{\varepsilon_1,0}}{\partial x}(\lambda_i,t)
=\frac{\partial S}{\partial x}(\lambda_i,t)+\varepsilon_1\dot f(\lambda_i),\\
& i=1,\cdots,m.
\end{split}
\end{align}
Hence $\mathcal{S}_{\varepsilon_1,0}(\lambda,t)=S(\lambda,t)\equiv0$, and  \eqref{eq40} tells us that
\begin{align*}
\dot{\mathcal{H}}_{\varepsilon_1,0}(\lambda)
&=\exp\left(\int^1_0\frac{\partial S}{\partial x}(\lambda,t)dt+\varepsilon_1\dot f(\lambda)\right)\\
&=\dot{\mathcal{H}}_{0,0}(\lambda)\cdot\exp\left(\varepsilon_1\dot f(\lambda)\right)\\
&=\exp\left(\varepsilon_1\dot f(\lambda)\right),
\end{align*}
where $\mathcal{H}_{\varepsilon_1,\varepsilon_2}$ represents the return map of \eqref{eq49}.
Observe that $\lambda_1,\cdots,\lambda_m$ are $m$ consecutive simple zeros of $f(x)$. We have
\begin{align}\label{eq43}
\text{sgn}\big(-\dot f(\lambda_1)\big)=\cdots=\text{sgn}\big((-1)^i\dot f(\lambda_i)\big)
=\cdots=\text{sgn}\big((-1)^m\dot f(\lambda_m)\big)\neq0.
\end{align}
Therefore, when $\varepsilon_1\neq0$ is small enough and $\varepsilon_2=0$, statement (a) holds, and \eqref{eq49} has a periodic solution near $x\equiv\lambda$.
Clearly, \eqref{eq49} is of the form \eqref{eq1}, and it satisfies Hypothesis {\bf (H)} by \eqref{eq13} as $\varepsilon_2=0$. According to statement (i) of Proposition \ref{lem5}, this periodic solution is hyperbolic.
As a result, statement (b) is obtained.


Now suppose that \eqref{eq1} has $s$ isolated periodic solutions, counted with multiplicities. From the above discussion, there exists a small $\varepsilon'_1\neq0$ such that \eqref{eq49}$|_{\varepsilon_1=\varepsilon'_1,\varepsilon_2=0}$ has at least $s$ hyperbolic periodic solutions.


On the other hand, we also get by \eqref{eq49} that
\begin{align*}
\mathcal{S}_{\varepsilon_1,\varepsilon_2}(\lambda_i,t)=S(\lambda_i,t)+\varepsilon_2\dot f(\lambda_i),\indent i=1,\cdots, m.
\end{align*}
Hence Hypothesis {\bf (H)} and \eqref{eq43} tell us that for either $\varepsilon_2>0$ or $\varepsilon_2<0$,
\begin{align*}
\text{sgn}\big(-\mathcal{S}_{\varepsilon_1,\varepsilon_2}(\lambda_1,t)\big)
=\cdots=\text{sgn}\big((-1)^i\mathcal{S}_{\varepsilon_1,\varepsilon_2}(\lambda_i,t)\big)
=\cdots=\text{sgn}\big((-1)^m\mathcal{S}_{\varepsilon_1,\varepsilon_2}(\lambda_m,t)\big)\neq0.
\end{align*}
This implies that for either $\varepsilon_2>0$ or $\varepsilon_2<0$, \eqref{eq49} satisfies Hypothesis {\bf (H)}, with all the periodic solutions contained in $\mathbb R\backslash\{\lambda_1,\cdots,\lambda_m\}$. Following statement (i) of Proposition \ref{lem5} and statement (ii) of Proposition \ref{lem3}, \eqref{eq49} has at most $m$ hyperbolic periodic solutions for either $\varepsilon_2>0$ or $\varepsilon_2<0$.

Based on the above, there exists a small $\varepsilon'_2\neq0$ such that
\begin{itemize}
\item[(c)] Equation \eqref{eq49}$|_{\varepsilon_1=\varepsilon'_1,\varepsilon_2=\varepsilon'_2}$ has at least $s$ hyperbolic periodic solutions.
\item[(d)] Equation \eqref{eq49}$|_{\varepsilon_1=\varepsilon'_1,\varepsilon_2=\varepsilon'_2}$ has at most $m$ hyperbolic periodic solutions.
\end{itemize}
We get $s\leq m$.

Finally, taking $S(x,t)=f(x)$, it is easy to know that $x_1(t)\equiv\lambda_1,\cdots,x_m(t)\equiv\lambda_m$ are $m$ isolated periodic solutions of the equation. Thus, the upper bound is sharp.

The proof of Theorem \ref{thm1} is finished.
\end{proof}

Here we give two concrete examples as a direct application of Theorem \ref{thm1}.

The first one shows again that the upper bound in the theorem is reachable. Consider equation
\begin{align}\label{eq16}
\frac{dx}{dt}=S(x,t)=x^4+x^3-13x^2-x+12.
\end{align}
Taking $\lambda_1=-2$, $\lambda_2=0$, $\lambda_3=2$ and $\lambda_4=4$, we have
\begin{align*}
S(\lambda_1,t)=-30<0,\ \ S(\lambda_2,t)=12>0, \ \ S(\lambda_3,t)=-18<0,\ \ S(\lambda_4,t)=120>0,
\end{align*}
which implies that $(-1)^i\cdot S(\lambda_i,t)>0$ for $i=1,2,3,4$. Thus Theorem \ref{thm1} tells us that \eqref{eq16} has at most $4$ isolated periodic solutions, counted with multiplicities. In fact, one can check that \eqref{eq16} has $4$ isolated periodic solutions $x_1(t)\equiv-4$, $x_2(t)\equiv-1$, $x_3(t)\equiv1$ and $x_4(t)\equiv3$. The upper bound is clearly reached.

The second one considers an equation where all the coefficients with respect to $x$ change signs
\begin{align}\label{eq17}
\begin{split}
\frac{dx}{dt}&=S(x,t)\\
&=\big(1+5\cos(2\pi t)\big)x^3-15\cos(2\pi t)x^2-\big(4+5\cos(2\pi t)\big)x+15\cos(2\pi t).
\end{split}
\end{align}
If we take $\lambda_1=-1$, $\lambda_2=1$ and $\lambda_3=3$, then
\begin{align}\label{eq18}
S(\lambda_1,t)=3>0,\ \ S(\lambda_2,t)=-3<0,\ \ S(\lambda_3,t)=15>0,
\end{align}
which means $(-1)^i\cdot S(\lambda_i,t)<0$ for $i=1,2,3$.
Following Theorem \ref{thm1}, \eqref{eq17} has at most $3$ isolated periodic solutions, counted with multiplicities. Furthermore, we can know by \eqref{eq18} that \eqref{eq17} has at least $2$ isolated periodic solutions $x_1(t)$ and $x_2(t)$ located in intervals $(-1,1)$ and $(1,3)$, respectively.


\section{Application on trigonometrical generalized Abel equation}\label{an
apprpriate label}
In this section we mainly prove Corollary \ref{example1} and Proposition \ref{prop1}. We also give an example, in which all the coefficients with respect to
$x$ change signs, to show the application of our result.

\begin{proof}[Proof of Corollary \ref{example1}]
It is clear that \eqref{eq28} can be rewritten in
\begin{align}\label{eq29}
\begin{split}
\frac{dx}{dt}&=S(x,t)\\
&=f_{\text{\bf{\em a}}}(x)+f_{\text{\bf{\em b}}}(x)\cos(2\pi t)+f_{\text{\bf{\em c}}}(x)\sin(2\pi t)\\
&=f_{\text{\bf{\em a}}}(x)+\sqrt{f_{\text{\bf{\em b}}}^2(x)+f_{\text{\bf{\em c}}}^2(x)}\sin\big(\theta(x)+2\pi t\big),
\end{split}
\end{align}
where
\begin{align}\label{eq19}
\theta(x)=
\left\{
\begin{aligned}
&\arcsin \left(f_{\text{\bf{\em b}}}(x)\bigg/\sqrt{f_{\text{\bf{\em b}}}^2(x)+f_{\text{\bf{\em c}}}^2(x)}\right),&f_{\text{\bf{\em b}}}^2(x)+f_{\text{\bf{\em c}}}^2(x)\neq0,\\
&0,&f_{\text{\bf{\em b}}}^2(x)+f_{\text{\bf{\em c}}}^2(x)=0.
\end{aligned}
\right.
\end{align}
By condition {\bf (C)}, we have
\begin{align}\label{eq20}
S(\kappa_i,t)=f_{\text{\bf{\em a}}}(\kappa_i)\bigg(1+\text{sgn}(f_{\text{\bf{\em a}}}(\kappa_i))\cdot\sin\big(\theta(\kappa_i)+2\pi t\big)\bigg),\indent i=1,\cdots,m.
\end{align}
Hence either $(-1)^i\cdot S(\kappa_i,t)\geq0$ for $i=1,\cdots,m$, or  $(-1)^i\cdot S(\kappa_i,t)\leq0$ for $i=1,\cdots,m$.
According to Theorem \ref{thm1}, \eqref{eq28} has at most $m$ isolated periodic solutions, counted with multiplicities.
\end{proof}

For instance, consider equation
\begin{align}\label{eq87}
\begin{split}
\frac{dx}{dt}&=S(x,t)\\
&=\big(3+5\cos(2\pi t)\big)x^4-\big(15+17\cos(2\pi t)\big)x^3\\
&\indent-2\big(3+10\cos(2\pi t)\big)x^2+\big(45+47\cos(2\pi t)\big)x+6\big(3+5\cos(2\pi t)\big).
\end{split}
\end{align}
Obviously, \eqref{eq87} is of the form \eqref{eq28}. It is easy to know by Corollary \ref{example1} that
\begin{align*}
&f_{\text{\bf{\em a}}}(x)=3x^4-15x^3-6x^2+45x+18,\\
&f_{\text{\bf{\em b}}}(x)=5x^4-17x^3-20x^2+47x+30,\\
&f_{\text{\bf{\em c}}}(x)=0.
\end{align*}
Hence,
\begin{align*}
f_{\text{\bf{\em a}}}^2(x)-f_{\text{\bf{\em b}}}^2(x)-f_{\text{\bf{\em c}}}^2(x)&=\big(f_{\text{\bf{\em a}}}(x)-f_{\text{\bf{\em b}}}(x)\big)\big(f_{\text{\bf{\em a}}}(x)+f_{\text{\bf{\em b}}}(x)\big).
\end{align*}
From a direct calculation,
\begin{align*}
f_{\text{\bf{\em a}}}(x)-f_{\text{\bf{\em b}}}(x)&=-2\left(x^4-x^3-7x^2+x+6\right)\\
&=-2(x+2)(x+1)(x-1)(x-3),
\end{align*}
and
\begin{align}\label{eq25}
f_{\text{\bf{\em a}}}(-2)=72>0,\ f_{\text{\bf{\em a}}}(-1)=-15<0,\ f_{\text{\bf{\em a}}}(1)=45>0,\ f_{\text{\bf{\em a}}}(3)=-63<0.
\end{align}
If we take
\begin{align*}
\kappa_1=-2,\ \kappa_2=-1,\ \kappa_3=1,\ \kappa_4=3,
\end{align*}
then $\kappa_1<\cdots<\kappa_4$ is $4$ zeros of $f_{\text{\bf{\em a}}}^2(x)-f_{\text{\bf{\em b}}}^2(x)-f_{\text{\bf{\em c}}}^2(x)$, and $(-1)^i f_{\text{\bf{\em a}}}(\kappa_i)<0$ for $i=1,\cdots,4$.
Following Corollary \ref{example1}, the number of isolated periodic solutions of \eqref{eq87} is no more than 4, counted with multiplicities.

Moreover, according to \eqref{eq20} in the proof of Corollary \ref{example1} and \eqref{eq25}, for equation \eqref{eq87} we know that
\begin{align}\label{eq27}
\begin{split}
&S(-2,t)\not\equiv0,\ \  S(-1,t)\not\equiv0,\ \ S(1,t)\not\equiv0,\ \ S(3,t)\not\equiv0,\\
&S(-2,t)\geq0,\ \  S(-1,t)\leq0,\ \ S(1,t)\geq0,\ \ S(3,t)\leq0.
\end{split}
\end{align}
Therefore \eqref{eq87} has at least $3$ isolated periodic solutions $x_1(t)$, $x_2(t)$ and $x_3(t)$ located in intervals $(-2,-1)$, $(-1,1)$ and
$(1,3)$, respectively.

\begin{rem}
As seen above, only the $4$ zeros of $f_{\text{\bf{\em a}}}(x)-f_{\text{\bf{\em b}}}(x)$ are used in the discussion for equation \eqref{eq87}. Observe that
\begin{align*}
f_{\text{\bf{\em a}}}(x)+f_{\text{\bf{\em b}}}(x)
&=2(2x+3)(2x+1)(x-2)(x-4).
\end{align*}
Then $-3/2$, $-1/2$, $2$ and $4$ are the other $4$ zeros of $f_{\text{\bf{\em a}}}^2(x)-f_{\text{\bf{\em b}}}^2(x)-f_{\text{\bf{\em c}}}^2(x)$. By a direct calculation,
\begin{align*}
f_{\text{\bf{\em a}}}(-\frac{3}{2})=\frac{45}{16}>0,\ \ f_{\text{\bf{\em a}}}(-\frac{1}{2})=-\frac{63}{16}<0,\ \ f_{\text{\bf{\em a}}}(2)=12>0,\ \ f_{\text{\bf{\em a}}}(4)=-90<0.
\end{align*}
Hence, for \eqref{eq87} these $4$ zeros also satisfy condition {\bf (C)}. In addition,
again following \eqref{eq20} in the proof of Corollary \ref{example1}, we get
\begin{align*}
&S(-\frac{3}{2},t)\not\equiv0,\ \  S(-\frac{1}{2},t)\not\equiv0,\ \ S(2,t)\not\equiv0,\ \ S(4,t)\not\equiv0,\\
&S(-\frac{3}{2},t)\geq0,\ \  S(-\frac{1}{2},t)\leq0,\ \ S(2,t)\geq0,\ \ S(4,t)\leq0.
\end{align*}
Together with \eqref{eq27}, the intervals where the isolated periodic solutions $x_1(t)$, $x_2(t)$ and $x_3(t)$ of \eqref{eq87} located can be improved to $(-3/2,-1)$, $(-1/2,1)$ and $(2,3)$, respectively.
\end{rem}

\begin{rem}
We emphasize that equation \eqref{eq87} is a generalized Abel equation which has no coefficients with respect to $x$ keeping signs.
\end{rem}

\begin{proof}[Proof of Proposition \ref{prop1}]
Firstly, if equation \eqref{eq28} has constant periodic solutions, then the assertion is valid. In what follows we prove the case that \eqref{eq28} has no constant periodic solutions, i.e. $S(x,t)\not\equiv0$ for arbitrary fixed $x\in\mathbb R$.

By assumption, suppose without loss of generality that $(-1)^i\cdot S(\lambda_i,t)\geq0$ for $i=1,\cdots,m$ (it is the same argument for the opposite case). Let $\mathcal{F}(x)=f_{\text{\bf{\em a}}}^2(x)-f_{\text{\bf{\em b}}}^2(x)-f_{\text{\bf{\em c}}}^2(x)$. We claim that
\begin{align}\label{eq36}
\mathcal{F}(\lambda_i)=f_{\text{\bf{\em a}}}^2(\lambda_i)-f_{\text{\bf{\em b}}}^2(\lambda_i)-f_{\text{\bf{\em c}}}^2(\lambda_i)\geq0,\ (-1)^i\cdot f_{\text{\bf{\em a}}}(\lambda_i)>0,\indent i=1,\cdots,m.
\end{align}
In fact, similar to the proof of Corollary \ref{example1}, $S(x,t)$ can be rewritten in
$$S(x,t)=f_{\text{\bf{\em a}}}(x)+\sqrt{f_{\text{\bf{\em b}}}^2(x)+f_{\text{\bf{\em c}}}^2(x)}\sin\big(\theta(x)+2\pi t\big),$$
where $\theta(x)$ is defined as in \eqref{eq19}. Then for each $i=1,\cdots,m$, we get
\begin{align*}
&|f_{\text{\bf{\em a}}}(\lambda_i)|\geq\sqrt{f_{\text{\bf{\em b}}}^2(\lambda_i)+f_{\text{\bf{\em c}}}^2(\lambda_i)}\
\text{(otherwise $S(\lambda_i,t)$ changes sign on $[0,1]$)},\ \text{i.e.}\
\mathcal{F}(\lambda_i)\geq0,\\
&(-1)^i\cdot f_{\text{\bf{\em a}}}(\lambda_i)\geq0.
\end{align*}
Furthermore, note that $f_{\text{\bf{\em b}}}^2(\lambda_i)+f_{\text{\bf{\em c}}}^2(\lambda_i)=0$ if $f_{\text{\bf{\em a}}}(\lambda_i)=0$, i.e.
$S(\lambda_i,t)\equiv0$ if $f_{\text{\bf{\em a}}}(\lambda_i)=0$. We have
$(-1)^i\cdot f_{\text{\bf{\em a}}}(\lambda_i)>0$ for $i=1,\cdots,m$.
Hence, \eqref{eq36} is true.

As a result, each open interval $(\lambda_i,\lambda_{i+1})$ contains an odd number of zeros (counted with multiplicities) of $f_{\text{\bf{\em a}}}(x)$, $i=1,\cdots,m-1$. Since $f_{\text{\bf{\em a}}}(x)$ is of degree no more than $m$, it actually has only one zero $\lambda'_i$ (counted with multiplicities) in $(\lambda_i,\lambda_{i+1})$, $i=1,\cdots,m-1$. Therefore, $f_{\text{\bf{\em a}}}(x)$ has exactly $m-1$ simple zeros $\lambda'_1<\cdots<\lambda'_{m-1}$ in $(\lambda_1,\lambda_m)$. If we take $I_1=(\lambda_1,\lambda'_1)$, $I_m=(\lambda'_{m-1},\lambda_m)$ and $I_{i}=(\lambda'_{i-1},\lambda'_{i})$ for $i=2,\cdots,m-1$, then
\begin{align}\label{eq38}
(-1)^i\cdot f_{\text{\bf{\em a}}}\left|_{I_i}\right.>0,\indent i=1,\cdots,m.
\end{align}

On the other hand, 
 again since $S(x,t)\not\equiv0$ for arbitrary fixed $x\in\mathbb R$,
 $f_{\text{\bf{\em a}}}(x)$ and $f_{\text{\bf{\em b}}}^2(x)+f_{\text{\bf{\em c}}}^2(x)$ have no common zeros. For this reason,
\begin{align}\label{eq37}
\mathcal{F}(\lambda'_i)=-\left(f_{\text{\bf{\em b}}}^2(\lambda'_i)+f_{\text{\bf{\em c}}}^2(\lambda'_i)\right)<0,\indent i=1,\cdots,m-1.
\end{align}

Now define $I'_i=[\lambda_i,\lambda'_i)$ for $i=1,\cdots,m-1$ and $I'_m=(\lambda'_{m-1},\lambda_m]$.
From \eqref{eq36} and \eqref{eq37}, there exists $\kappa_i\in I'_i$ such that $\mathcal{F}(\kappa_i)=0$ for $i=1,\cdots,m$. Furthermore, noting that
$I'_i\subseteq \overline{I_i}$, we know by \eqref{eq38} that
\begin{align*}
(-1)^i\cdot f_{\text{\bf{\em a}}}(\kappa_i)\geq0,\indent i=1,\cdots,m.
\end{align*}
Thus, condition {\bf (C)} holds, and the proof of Proposition \ref{prop1} is finished.
\end{proof}

\section{Geometric hypotheses for generalized Abel equation}\label{an
apprpriate label}

Firstly, let $a(t), b(t)$ be two smooth functions on interval $[0,1]$ with $a(0)=a(1)$, $b(0)=b(1)$ and $a(t)>0$. Let $\gamma_{\lambda}$ be defined as in section 1, i.e. $\gamma_{\lambda}:\big(t,\lambda a(t)+b(t)\big)$ is a family of curves lying in $[0,1]\times\mathbb R$ with parameter $\lambda\in\mathbb R$.

We know that \eqref{eq1} has an invariant form under the transformation $y=\big(x-b(t)\big)/a(t)$. In fact, this transformation changes \eqref{eq1} into
\begin{align}\label{eq31}
\begin{split}
\frac{dy}{dt}&=\wt S(y,t)\\
&=\frac{1}{a(t)}\left(S\big(a(t)y+b(t),t\big)-\dot a(t)y-\dot b(t)\right).
\end{split}
\end{align}
Since
\begin{align*}
S\big(a(t)y+b(t),t\big)
=\sum^m_{i=0}a_i(t)\big(a(t)y+b(t)\big)^i,
\end{align*}
$\wt S(y,t)$ is a polynomial in $y$ of degree no more than $m$. Therefore \eqref{eq31} is of the form \eqref{eq1}.


Now consider the vector field $v_S=\big(1,S(x,t)\big)$ induced by \eqref{eq1}. One can check by \eqref{eq31} that
\begin{align}\label{eq26}
\begin{split}
&\det\left(\dot{\gamma_{\lambda}}, v_S\right)\left|_{\gamma_{\lambda_{}}}\right.=a(t)\cdot\wt S(\lambda,t),
\end{split}
\end{align}
where $\det(\cdot,\cdot)$ is the determinant of two $2$-dimensional vector fields. The proof of Theorem \ref{thm2} is given below.

\begin{proof}[Proof of Theorem \ref{thm2}]
According to the discussion above, we transform \eqref{eq1} into \eqref{eq31}.
Since equality \eqref{eq26} shows that
\begin{align*}
(-1)^i\cdot\wt S(\lambda_i,t)
=\frac{(-1)^i\cdot\det\left(\dot{\gamma_{\lambda_i}}, v_S\right)\left|_{\gamma_{\lambda_i}}\right.}{a(t)},\indent i=1,\cdots,m,
\end{align*}
the assumption implies that either
$(-1)^i\cdot\wt S(\lambda_i,t)\geq0$ for $i=1,\cdots,m$, or $(-1)^i\cdot\wt S(\lambda_i,t)\leq0$ for $i=1,\cdots,m$.
Hence, \eqref{eq31} satisfies Hypothesis {\bf (H)}. Using Theorem \ref{thm1}, \eqref{eq31} has at most $m$ isolated periodic solutions (counted with multiplicities), and the upper bound is sharp. Noting that $a(0)=a(1)$ and $b(0)=b(1)$, equations \eqref{eq1} and \eqref{eq31} have the same number of periodic solutions. As a result, the assertion of Theorem \ref{thm2} is obtained.
\end{proof}

For instance, consider equation
\begin{align}\label{eq93}
\begin{split}
\frac{dx}{dt}
&= S(x,t)\\
&=\cos(2\pi t)x^3
+\bigg(2+9\cos(2\pi t)+3\cos^2(2\pi t)\bigg)x^2\\
&\indent-\bigg(36\cos(2\pi t)+24\cos^2(2\pi t)+4\cos^3(2\pi t)\bigg)x-1.
\end{split}
\end{align}
Taking $a(t)=3+\cos(2\pi t)$ and $b(t)\equiv0$, we have $\gamma_{\lambda}=\big(t,\lambda a(t)\big)$ and
\begin{align}\label{eq94}
S(x,t)
&=\cos(2\pi t)x^3
+\big(2+3\cos(2\pi t)\cdot a(t)\big)x^2
-\big(4\cos(2\pi t)\cdot a^2(t)\big)x-1.
\end{align}
By the previous study and a direct calculation,
\begin{align*}
\det\left(\dot{\gamma_{\lambda}}, v_S\right)\big|_{\gamma_{\lambda}}
&=\lambda
\bigg(\cos(2\pi t)\cdot a^3(t)(\lambda-1)(\lambda+4)+2\lambda\cdot a^2(t)-\dot{a}(t)\bigg)-1.
\end{align*}
Therefore,
\begin{align*}
&\det\left(\dot{\gamma_{1}}, v_S\right)\big|_{\gamma_{1}}
=2\big(3+\cos(2\pi t)\big)^2+2\pi\cdot\sin(2\pi t)-1>0,\\
&\det\left(\dot{\gamma_{0}}, v_S\right)\big|_{\gamma_{0}}
=-1<0,\\
&\det\left(\dot{\gamma_{-4}}, v_S\right)\big|_{\gamma_{-4}}
=32\big(3+\cos(2\pi t)\big)^2-8\pi\cdot\sin(2\pi t)-1>0.
\end{align*}
Since $-4<0<1$, \eqref{eq93} satisfies Hypothesis {\bf (H')} with $\lambda_1=-4$, $\lambda_2=0$ and $\lambda_3=1$. We get from Theorem \ref{thm2} that \eqref{eq93} has at most 3 isolated periodic solutions,
counted with multiplicities.

\begin{proof}[Proof of Corollary \ref{cor1}]
It is known that two planar vectors are positively oriented (resp. negatively oriented) if and only if the determinant is positive (resp. negative).
By assumption, if $\{v_S,\dot{\gamma_{\lambda_i}}\}$ and $\{v_S,\dot{\gamma_{\lambda_{i+1}}}\}$ have opposite orientations for each $i=1,\cdots,m-1$, then \eqref{eq1} satisfies Hypothesis {\bf (H')}. Our assertion follows immediately from Theorem \ref{thm2}.
\end{proof}

\end{document}